\documentclass[letterpaper,final,twoside,leqno]{amsart}

\usepackage[normalem]{ulem}
\usepackage{mathtools}
\usepackage{mathrsfs}  
\usepackage[arrow,curve,matrix]{xy}
\usepackage[retainorgcmds]{IEEEtrantools} 
\usepackage{cite}

\usepackage{scalerel}

\usepackage{comment}

\newif\ifdraft
\draftfalse

\numberwithin{figure}{section}

\usepackage{amsmath,amsfonts,amsthm,mathrsfs}
\usepackage{amssymb}
\usepackage{mathrsfs}

\DeclareFontFamily{OMS}{rsfs}{\skewchar\font'60}
\DeclareFontShape{OMS}{rsfs}{m}{n}{<-5>rsfs5 <5-7>rsfs7 <7->rsfs10 }{}
\DeclareSymbolFont{rsfs}{OMS}{rsfs}{m}{n}
\DeclareSymbolFontAlphabet{\scr}{rsfs}

\usepackage[bookmarks]{hyperref}
\usepackage[usenames,dvipsnames]{xcolor}
\hypersetup{
    colorlinks=true,
    linkcolor={blue!50!black},
    citecolor={green!50!black},
    urlcolor={red!80!black}
}

\usepackage[alphabetic,initials]{amsrefs}

\usepackage{enumitem}

\usepackage{chngcntr}

\ifdraft
\usepackage[notcite,notref,color]{showkeys}

\definecolor{labelkey}{gray}{0.5}
\fi

\usepackage{tikz}

\usepackage{tikz-cd}
\tikzset{commutative diagrams/arrow style=math font}

\usetikzlibrary{matrix,arrows}
\newlength{\myarrowsize} 

\pgfarrowsdeclare{cmto}{cmto}{
	\pgfsetdash{}{0pt} 
	\pgfsetbeveljoin 
	\pgfsetroundcap 
	\setlength{\myarrowsize}{0.6pt}
	\addtolength{\myarrowsize}{.5\pgflinewidth}
	\pgfarrowsleftextend{-4\myarrowsize-.5\pgflinewidth} 
	\pgfarrowsrightextend{.8\pgflinewidth}
}{
	\setlength{\myarrowsize}{0.6pt} 
  	\addtolength{\myarrowsize}{.5\pgflinewidth}  
	\pgfsetlinewidth{0.5\pgflinewidth}
	\pgfsetroundjoin
	\pgfpathmoveto{\pgfpoint{1.5\pgflinewidth}{0}}
	\pgfpatharc{-109}{-170}{4\myarrowsize}
	\pgfpatharc{10}{189}{0.58\pgflinewidth and 0.2\pgflinewidth}
	\pgfpatharc{-170}{-115}{4\myarrowsize+\pgflinewidth}
	\pgfpathclose
	\pgfusepathqfillstroke
	\pgfpathmoveto{\pgfpoint{1.5\pgflinewidth}{0}}
	\pgfpatharc{109}{170}{4\myarrowsize}
	\pgfpatharc{-10}{-189}{0.58\pgflinewidth and 0.2\pgflinewidth}
	\pgfpatharc{170}{115}{4\myarrowsize+\pgflinewidth}
	\pgfpathclose
	\pgfusepathqfillstroke
	\pgfsetlinewidth{2\pgflinewidth}
}

\pgfarrowsdeclare{cmonto}{cmonto}{
	\pgfsetdash{}{0pt} 
	\pgfsetbeveljoin 
	\pgfsetroundcap 
	\setlength{\myarrowsize}{0.6pt}
	\addtolength{\myarrowsize}{.5\pgflinewidth}
	\pgfarrowsleftextend{-4\myarrowsize-.5\pgflinewidth} 
	\pgfarrowsrightextend{.8\pgflinewidth}
}{
	\setlength{\myarrowsize}{0.6pt} 
  	\addtolength{\myarrowsize}{.5\pgflinewidth}  
	\pgfsetlinewidth{0.5\pgflinewidth}
	\pgfsetroundjoin
	\pgfpathmoveto{\pgfpoint{1.5\pgflinewidth}{0}}
	\pgfpatharc{-109}{-170}{4\myarrowsize}
	\pgfpatharc{10}{189}{0.58\pgflinewidth and 0.2\pgflinewidth}
	\pgfpatharc{-170}{-115}{4\myarrowsize+\pgflinewidth}
	\pgfpathclose
	\pgfusepathqfillstroke
	\pgfpathmoveto{\pgfpoint{1.5\pgflinewidth}{0}}
	\pgfpatharc{109}{170}{4\myarrowsize}
	\pgfpatharc{-10}{-189}{0.58\pgflinewidth and 0.2\pgflinewidth}
	\pgfpatharc{170}{115}{4\myarrowsize+\pgflinewidth}
	\pgfpathclose
	\pgfusepathqfillstroke
	\pgfpathmoveto{\pgfpoint{1.5\pgflinewidth-0.3em}{0}}
	\pgfpatharc{-109}{-170}{4\myarrowsize}
	\pgfpatharc{10}{189}{0.58\pgflinewidth and 0.2\pgflinewidth}
	\pgfpatharc{-170}{-115}{4\myarrowsize+\pgflinewidth}
	\pgfpathclose
	\pgfusepathqfillstroke
	\pgfpathmoveto{\pgfpoint{1.5\pgflinewidth-0.3em}{0}}
	\pgfpatharc{109}{170}{4\myarrowsize}
	\pgfpatharc{-10}{-189}{0.58\pgflinewidth and 0.2\pgflinewidth}
	\pgfpatharc{170}{115}{4\myarrowsize+\pgflinewidth}
	\pgfpathclose
	\pgfusepathqfillstroke
	\pgfsetlinewidth{2\pgflinewidth}
}

\pgfarrowsdeclare{cmhook}{cmhook}{
	\pgfsetdash{}{0pt} 
	\pgfsetbeveljoin 
	\pgfsetroundcap 
	\setlength{\myarrowsize}{0.6pt}
	\addtolength{\myarrowsize}{.5\pgflinewidth}
	\pgfarrowsleftextend{-4\myarrowsize-.5\pgflinewidth} 
	\pgfarrowsrightextend{.8\pgflinewidth}
}{
	\setlength{\myarrowsize}{0.6pt} 
  	\addtolength{\myarrowsize}{.5\pgflinewidth}  
 	\pgfsetdash{}{0pt}
	\pgfsetroundcap
	\pgfpathmoveto{\pgfqpoint{0pt}{-4.667\pgflinewidth}}
	\pgfpathcurveto
    {\pgfqpoint{4\pgflinewidth}{-4.667\pgflinewidth}}
    {\pgfqpoint{4\pgflinewidth}{0pt}}
    {\pgfpointorigin}
	\pgfusepathqstroke
}

\newenvironment{diagram*}[2]{%
\[%
\begin{tikzpicture}[>=cmto,baseline=(current bounding box.center),%
	to/.style={->,font=\scriptsize,cap=round},%
	into/.style={cmhook->,font=\scriptsize,cap=round},%
	onto/.style={-cmonto,font=\scriptsize,cap=round},%
	math/.style={matrix of math nodes, row sep=#2, column sep=#1,%
		text height=1.5ex, text depth=0.25ex}]%
}{%
\end{tikzpicture}%
\]%
\ignorespacesafterend%
}

\DeclareMathOperator{\pr}{pr}

\newcommand{\wtilde}{\widetilde}

\newcommand{\hooklongrightarrow}{\lhook\joinrel\longrightarrow}

\newcommand{\theoremref}[1]{\hyperref[#1]{Theorem~\ref*{#1}}}
\newcommand{\lemmaref}[1]{\hyperref[#1]{Lemma~\ref*{#1}}}
\newcommand{\definitionref}[1]{\hyperref[#1]{Definition~\ref*{#1}}}
\newcommand{\propositionref}[1]{\hyperref[#1]{Proposition~\ref*{#1}}}
\newcommand{\conjectureref}[1]{\hyperref[#1]{Conjecture~\ref*{#1}}}
\newcommand{\corollaryref}[1]{\hyperref[#1]{Corollary~\ref*{#1}}}
\newcommand{\exampleref}[1]{\hyperref[#1]{Example~\ref*{#1}}}

\makeatletter
\let\old@caption\caption
\renewcommand*{\caption}[1]{%
  \setcounter{figure}{\value{equation}}%
  \stepcounter{equation}%
  \old@caption{#1}\relax%
}
\makeatother

\newcounter{intro}

\newtheorem{intro-conjecture}[intro]{Conjecture}
\newtheorem{intro-corollary}[intro]{Corollary}
\newtheorem{intro-theorem}[intro]{Theorem}

\newcommand{\parref}[1]{\hyperref[#1]{\S\ref*{#1}}}

\makeatletter
\newcommand*\if@single[3]{%
  \setbox0\hbox{${\mathaccent"0362{#1}}^H$}%
  \setbox2\hbox{${\mathaccent"0362{\kern0pt#1}}^H$}%
  \ifdim\ht0=\ht2 #3\else #2\fi
  }
\newcommand*\rel@kern[1]{\kern#1\dimexpr\macc@kerna}
\newcommand*\widebar[1]{\@ifnextchar^{{\wide@bar{#1}{0}}}{\wide@bar{#1}{1}}}
\newcommand*\wide@bar[2]{\if@single{#1}{\wide@bar@{#1}{#2}{1}}{\wide@bar@{#1}{#2}{2}}}
\newcommand*\wide@bar@[3]{%
  \begingroup
  \def\mathaccent##1##2{%
    \if#32 \let\macc@nucleus\first@char \fi
    \setbox\z@\hbox{$\macc@style{\macc@nucleus}_{}$}%
    \setbox\tw@\hbox{$\macc@style{\macc@nucleus}{}_{}$}%
    \dimen@\wd\tw@
    \advance\dimen@-\wd\z@
    \divide\dimen@ 3
    \@tempdima\wd\tw@
    \advance\@tempdima-\scriptspace
    \divide\@tempdima 10
    \advance\dimen@-\@tempdima
    \ifdim\dimen@>\z@ \dimen@0pt\fi
    \rel@kern{0.6}\kern-\dimen@
    \if#31
      \overline{\rel@kern{-0.6}\kern\dimen@\macc@nucleus\rel@kern{0.4}\kern\dimen@}%
      \advance\dimen@0.4\dimexpr\macc@kerna
      \let\final@kern#2%
      \ifdim\dimen@<\z@ \let\final@kern1\fi
      \if\final@kern1 \kern-\dimen@\fi
    \else
      \overline{\rel@kern{-0.6}\kern\dimen@#1}%
    \fi
  }%
  \macc@depth\@ne
  \let\math@bgroup\@empty \let\math@egroup\macc@set@skewchar
  \mathsurround\z@ \frozen@everymath{\mathgroup\macc@group\relax}%
  \macc@set@skewchar\relax
  \let\mathaccentV\macc@nested@a
  \if#31
    \macc@nested@a\relax111{#1}%
  \else
    \def\gobble@till@marker##1\endmarker{}%
    \futurelet\first@char\gobble@till@marker#1\endmarker
    \ifcat\noexpand\first@char A\else
      \def\first@char{}%
    \fi
    \macc@nested@a\relax111{\first@char}%
  \fi
  \endgroup
}
\makeatother

\usepackage{amsbsy}

\usepackage{marvosym}

\DeclareMathAlphabet{\smallchanc}{OT1}{pzc}%
                                 {m}{it}
\DeclareFontFamily{OT1}{pzc}{}
\DeclareFontShape{OT1}{pzc}{m}{it}%
             {<-> s * [1.100] pzcmi7t}{}
\DeclareMathAlphabet{\mathchanc}{OT1}{pzc}%
                                 {m}{it}

\newcommand{\sE}{\mathscr{E}}
\newcommand{\sF}{\mathscr{F}}

\newcommand{\sO}{\mathscr{O}}

\newcommand{\sQ}{\mathscr{Q}}

\newcommand{\bC}{\mathbb{C}}

\newcommand{\bN}{\mathbb{N}}

\newcommand{\bP}{\mathbb{P}}
\newcommand{\bQ}{\mathbb{Q}}
\newcommand{\bR}{\mathbb{R}}

\newcommand{\bZ}{\mathbb{Z}}

\DeclareSymbolFont{largesymbolsA}{U}{jkpexa}{m}{n}
\SetSymbolFont{largesymbolsA}{bold}{U}{jkpexa}{bx}{n}
\DeclareMathSymbol{\varprod}{\mathop}{largesymbolsA}{16}
\makeatletter
\newcommand{\LeftEqNo}{\let\veqno\@@leqno}
\makeatother
\newcommand{\properideal}%
        {\subsetneq}

\DeclareMathOperator{\rank}{{rank}}

\DeclareMathOperator{\red}{red}

\DeclareMathOperator{\supp}{{supp}}

\newcommand{\factor}[2]{\left. \raise .2em\hbox{\ensuremath{#1}\vphantom{$I^d$}}
\hskip -.1em \right/ \hskip -.4em \raise -.3em\hbox{\ensuremath{#2}}}%
\newcommand\mtimes[3]{{\varprod_{#1}^{#2}}_{\raise 1ex \hbox{\scriptsize #3}}}%

\def\dimcoh#1.#2.#3.{h^{#1}(#2,#3)}
\def\hypcoh#1.#2.#3.{\mathbb H_{\vphantom{l}}^{#1}(#2,#3)}
\def\loccoh#1.#2.#3.#4.{H^{#1}_{#2}(#3,#4)}
\def\dimloccoh#1.#2.#3.#4.{h^{#1}_{#2}(#3,#4)}
\def\lochypcoh#1.#2.#3.#4.{\mathbb H^{#1}_{#2}(#3,#4)}
\def\seslong#1.#2.#3.{0  \longrightarrow  #1   \longrightarrow 
 #2 \longrightarrow #3 \longrightarrow 0} 
\def\sesshort#1.#2.#3.{0
 \rightarrow #1 \rightarrow #2 \rightarrow #3 \rightarrow 0}
\def\Iff#1#2#3{
\hfil\hbox{\hsize =#1
\vtop{\noin #2}
\hskip.5cm 
\lower.5\baselineskip\hbox{$\Leftrightarrow$}\hskip.5cm
\vtop{\noin #3}}\hfil\medskip}
\newcommand{\union}\cup
\newcommand{\intersect}\cap
\newcommand{\Union}\bigcup
\newcommand{\Intersect}\bigcap
\def\myoplus#1.#2.{\underset #1 \to {\overset #2 \to \oplus}}

\newcommand{\resto}[1]{\raise -.5ex\hbox{$\vert$}_{#1}}

\newcommand\noin{\noindent}

   [2017/06/13 v0.1 finetuning mabliautoref and theorem setup]

\RequirePackage{hyperref}

\newcommand{\sectionsize}{} 
\newcommand{\theoremsize}{} 

\renewcommand{\subsectionautorefname}{\sectionsize\sf \subsectionautorefname}

\makeatletter
\@ifdefinable\equationname{\let\equationname\equationautorefname}
\def\equationautorefname~#1\@empty\@empty\null{\protect{\theoremsize
    (#1\@empty\@empty\null)}}%
\@ifdefinable\AMSname{\let\AMSname\AMSautorefname}
\def\AMSautorefname~#1\@empty\@empty\null{
  ( #1\@empty\@empty\null)}%
\@ifdefinable\itemname{\let\itemname\itemautorefname}
\def\itemautorefname~#1\@empty\@empty\null{\theoremsize{%
    {#1}}\@empty\@empty\null%
}%
\makeatother

\RequirePackage{amsthm}
\RequirePackage{aliascnt}
\newcommand{\basetheorem}[3]{%
    \newtheorem{#1}{#2}[#3]
    \newtheorem*{#1*}{#2}
    \expandafter\def\csname #1autorefname\endcsname{#2}
}%
\newcommand{\maketheorem}[3]{%
    \newaliascnt{#1}{#2}
    \newtheorem{#1}[#1]{\theoremsize #3}
    \aliascntresetthe{#1}
    \expandafter\def\csname #1autorefname\endcsname{#3}
    \newtheorem{#1*}{#3}
}%
\newcommand{\baseremark}[3]{%
    \newtheorem{#1}{#2}{#3}
    \newtheorem*{#1*}{#2}
    \expandafter\def\csname #1autorefname\endcsname{#2}
}%
\newcommand{\makeremark}[3]{%
    \newaliascnt{#1}{#2}
    \newtheorem{#1}[equation]{#3}
    \aliascntresetthe{#1}
    \expandafter\def\csname #1autorefname\endcsname{\theoremsize\sf #3}
    \newtheorem{#1*}{#3}
}%

\theoremstyle{plain}   

\basetheorem{theorem}{Theorem}{section}

\theoremstyle{definition}    

\newcommand{\hide}[1]{}

\setlist[enumerate, 1]{font=\upshape}

\numberwithin{figure}{section}

\DeclareMathOperator{\Div}{Div}
\DeclareMathOperator{\Gal}{Gal}
\DeclareMathOperator{\Mov}{Mov}
\DeclareMathOperator{\N}{N}
\DeclareMathOperator{\nl}{NL}
\DeclareMathOperator{\NE^1}{NE^1}
\DeclareMathOperator{\Pic}{Pic}

\theoremstyle{plain}
\maketheorem{prop}{theorem}{Proposition}
\maketheorem{cor}{theorem}{Corollary}
\maketheorem{lem}{theorem}{Lemma}
\maketheorem{conj}{theorem}{Conjecture}
\maketheorem{ex}{theorem}{Example}

\newtheorem{proposition}[prop]{Proposition}
\newtheorem{corollary}[cor]{Corollary}
\newtheorem{lemma}[lem]{Lemma}
\newtheorem{example}[ex]{Example}

\theoremstyle{definition}
\maketheorem{defini}{theorem}{Definition}
\newtheorem{definition}[defini]{Definition}
\theoremstyle{remark}
\maketheorem{rem}{theorem}{Remark}
\newtheorem{remark}[rem]{Remark}
\maketheorem{notat}{theorem}{Notation}
\maketheorem{notr}{theorem}{Notation-Remark}
\maketheorem{defnot}{theorem}{Definition-Notation}
\newtheorem{notation}[notat]{Notation}
\newtheorem{not-rem}[notr]{Notation-Remark}
\newtheorem{def-not}[defnot]{Definition-Notation}

\maketheorem{set}{theorem}{Set-up}
\newtheorem{set-up}[set]{Set-up}

\maketheorem{term}{theorem}{Terminology}

\maketheorem{fac}{theorem}{Fact}

\maketheorem{ass}{theorem}{Assumption}

\maketheorem{clai}{theorem}{Claim}
\newtheorem{claim}[clai]{Claim}

\maketheorem{sclai}{theorem}{Subclaim}

\setlist[enumerate]{label=(\thetheorem.\arabic*), before={\setcounter{enumi}{\value{equation}}}, after={\setcounter{equation}{\value{enumi}}}}

\numberwithin{equation}{theorem}

\makeatletter
\let\amsmath@bigm\bigm

\renewcommand{\bigm}[1]{%
  \ifcsname fenced@\string#1\endcsname
    \expandafter\@firstoftwo
  \else
    \expandafter\@secondoftwo
  \fi
  {\expandafter\amsmath@bigm\csname fenced@\string#1\endcsname}%
  {\amsmath@bigm#1}%
}
\newcommand{\DeclareFence}[2]{\@namedef{fenced@\string#1}{#2}}
\makeatother

\DeclareFence{\mid}{|}

\begin{document}

\title[Chern class inequalities for non-uniruled varieties]
{Chern class inequalities for non-uniruled projective varieties} 

\author{Erwan Rousseau}
\address{Erwan Rousseau \\ Univ Brest\\
		CNRS UMR 6205
		\\Laboratoire de Mathematiques de Bretagne Atlantique\\ F-29200 Brest, France
		}
\email{\href{mailto:erwan.rousseau@univ-brest.fr}{erwan.rousseau@univ-brest.fr}}
\urladdr{\href{http://eroussea.perso.math.cnrs.fr/}{http://eroussea.perso.math.cnrs.fr/}}

\author{Behrouz Taji}
\address{Behrouz Taji, School of Mathematics and Statistics - Red Centre,
The University of New South Wales, NSW 2052 Australia}
\email{\href{mailto:b.taji@unsw.edu.au}{b.taji@unsw.edu.au}}
\urladdr{\href{https://web.maths.unsw.edu.au/~btaji/}{https://web.maths.unsw.edu.au/~btaji/}}




\setlength{\parskip}{0.19\baselineskip}


\begin{abstract}
 It is known that projective minimal models satisfy the celebrated 
 Miyaoka--Yau inequalities. 
 In this article, we extend these inequalities 
 to the set of all smooth, projective and non-uniruled 
 varieties. 
 \end{abstract}

\maketitle


\newcommand\hmarginpar[1]{}

\section{Introduction}
\label{subsect:intro}

Interconnections between topological, analytic and algebraic structures of compact complex 
varieties is a central theme in various branches of geometry and topology. 
Many of the classical results in this area use characteristic classes, in particular Chern classes to 
describe such connections. 
Apart from the celebrated Hizerbruch--Riemann--Roch theorems, 
prominent examples were found by Bogomolov \cite{Bog79} and---in a different direction---Yau \cite{MR0451180}, as 
a consequence of his (and Aubin's) solution to Calabi's conjecture. More precisely,  
he established that an $n$-dimensional 
compact K\"ahler manifold $(X,w)$ with $c_1(X)<0$ satisfies the inequality 
\begin{equation}\label{Yau}
\int_X  \big( 2(n+1) c_2(X)  - n c_1^2(X)  \big) \wedge w^{n-2}  \geq 0 .
\end{equation}
In a more general setting, using his generic semipositivity result, Miyaoka \cite{Miyaoka87} 
showed that 
any \emph{minimal variety}\footnote{Here minimal is in the sense of the minimal model program, that is $K_X$ is assumed to be 
nef, with $X$ having only terminal singularities.} $X$ satisfies the inequality  
\begin{equation}\label{Miyaoka}
\big( 3 c_2 (X) - c_1^2(X) \big)  \cdot H^{n-2}  \geq 0,
\end{equation}
for every ample divisor $H\subset X$. The combination of the two inequalities \autoref{Yau}, \autoref{Miyaoka}
and their analogues 
are nowadays referred to as the 
\emph{Miyaoka--Yau inequalities}.

 The purpose of this article is to establish that, as long as $X$ is not covered by rational curves, 
 it satisfies a Chern class inequality generalizing \autoref{Miyaoka}. 
 Throughout this paper all varieties will be over $\bC$.

\begin{theorem}\label{thm:main}
Let $X$ be a smooth, projective and non-uniruled variety of dimension $n$, 
and $H$ any ample divisor. There is a decomposition $K_X = P+N$ into $\bQ$-divisors, 
with $P\cdot N \cdot H^{n-2}=0$, 
such that
\begin{equation}\label{eq:MY1}
\big( 3c_2(X) - c_1^2(X)  \big)\cdot H^{n-2} \geq N^2 \cdot H^{n-2}  .
\end{equation}
Moreover, we have $N=0$, when $K_X$ is nef. 
\end{theorem}

A few remarks about the statement of \autoref{thm:main}. 
First, we note that it is well-known that, for a non-uniruled variety, Chern class inequalities of the form 
\autoref{eq:MY1} cannot be gleaned 
from the ones for its minimal models, when they exist. 
Second, as is evident from the statement of the theorem, the quantity on the right-hand-side 
of \autoref{eq:MY1} (which we may 
think of as an error term) 
is forced on us by the \emph{negative part} of  
Zariski decomposition. More precisely, given a compete intersection surface $S\subset X$ defined by 
very general members of very ample linear systems, 
the divisor $N$ is an extension of the negative part of the Zariski decomposition for $K_X|_S$
(see \autoref{sect:Section3-Moving} for the details). This 
extension is in the sense of the Noether-Lefschetz-type theorems (\autoref{prop:Moish}). 

\autoref{thm:main} is a special case of the following more general result that we obtain in this article, 
 which is in fact needed for the proof of \autoref{thm:main}. 

\begin{theorem}\label{thm:main2}
Let $(X,D)$ be a log-smooth pair of dimension $n$, with $D$ being a rational divisor. 
Assume that $H$ is an ample divisor. 
If $K_X+D$ is pseudo-effective, then, there is a decomposition 
$K_X+D = P+N$, that is $H$-orthogonal in the sense that $P\cdot N\cdot H^{n-2 }=0$, and 
for which the inequality 
\begin{equation}\label{eq:MY2}
\big(  3\widehat c_2(X,D)  - \widehat c_1^2 (X,D)   \big)  \cdot H^{n-2}  \geq  N^2\cdot H^{n-2} 
\end{equation}
holds. Furthermore, when $K_X+D$ is nef, we have $N=0$.
\end{theorem}

The Chern classes $\widehat c_i(\cdot)$ 
in \autoref{thm:main2} are in the sense of orbifolds (see \autoref{eq:OC}), and 
when $D$ is reduced, they coincide with the usual notion of Chern classes.

\subsection{General strategy of the proof} 
For simplicity we will focus mostly on \autoref{thm:main}; the case where $D=0$.

As was observed by Miyaoka \cite{Miyaoka87} and later on Simpson \cite{MR944577}, 
it is sometimes possible 
to use the Bogomolov 
inequality \cite{Bog79} to establish Miyaoka--Yau inequalities. 
But the Bogomolov inequality is generally valid  
when the polarization is defined by ample or nef divisors, which is applicable---for the purpose of Miyaoka--Yau 
inequalities---when 
the variety is minimal. But for a general non-uniruled variety $X$ no such polarization exists. 
On the other hand, we show in the current article that, thanks to the result of Boucksom--Demailly--P\u{a}un--Peternell 
\cite{BDPP}, 
after cutting down by hyperplanes, 
the above divisor $P$ defines a so-called 
\emph{movable cycle}; a potentially natural choice for a polarization.  
But in general there is no topological Bogomolov-type inequality 
for sheaves that are semistable with respect to a movable class $\gamma$. 
At best, assuming that $\sE$ is locally free, one can use 
a Gauduchon metric $w_G$ constructed in \cite[Append.]{CP11}, 
with $w_G^{n-1}\equiv \gamma$, and Li--Yau's result on the 
existence of Hermitian-Einstein metrics \cite{LY15}
to establish the Bogomolov inequality with respect to $w_G^{n-2}$. But since $w_G^{n-2}$ is not closed, 
this would not yield a topological inequality. 
However, thanks to a fundamental 
result of Langer \cite[Thm.~3.4]{Langer04a}, 
semistability with respect to 
a certain subset 
of movable classes \emph{does} lead to the classical Bogomolov inequality. 
Having this important fact in mind, we use the definition 
of the Zariski decomposition to show that the intersection of $P$ with $H^{n-2}$ belongs to 
this smaller subset, as long as $X$ is of general type. 
With this observation, and using further properties of $P$, we then 
show that, thanks to Campana-P\u{a}un's result on positivity properties of the (log-)cotangent sheaf with 
respect to movable cycles \cite{CP16}, much of Miyaoka's original approach can then be 
adapted to establish \autoref{eq:MY1}.

In the more general setting of non-uniruled varieties, that is when $K_X$ is pseudo-effective
\cite{BDPP}, given 
an ample divisor $A$ and any $m\in \bN$,
we consider the pair $(X, \frac{1}{m} A)$, which is now of log-general type. Here, the log-version of 
\autoref{thm:main} is needed, forcing us to resort to orbifolds (in the sense of Campana) 
and their Chern classes as was defined in \cite[Sect.~2]{GT16}, following \cite{MR717614}. 
With the inequality \autoref{eq:MY2} at hand, one can then extract the inequality \autoref{eq:MY1} 
through a limiting 
process, which 
is  
reminiscent of \cite{GT16}, but employed for somewhat different reasons. 

\subsection{Related results}
Chern class inequalities for surfaces and their connection to the Zariski decomposition 
was first studied by Miyaoka in \cite{Mi84} and later on by Wahl \cite{Wah94}, 
Megyesi \cite{Meg99}, Langer \cite{La01} and others. 
In higher dimensions, when $K_X+D$ is movable and $\dim X=3$, 
the inequality \autoref{eq:MY2} is established in \cite{RT16}
for (mildly) singular pairs. 
Under the assumption that $K_X+D$ is nef and big, such Chern class inequalities have a rich 
history and were discovered by 
Kobayashi \cite{Kob84}, Tsuji \cite{MR0944606}, Tian \cite{Tia94}, 
to name a few. More recently, and in a more general setting, they have been studied in joint papers 
with Greb--Kebekus--Peternell \cite{GKPT15} and with Guenancia \cite{GT16}.
Further results have been established by Deng \cite{Deng21} and Hai--Schreieder \cite{HS20}.  
Finally, we note that the methods that we use in this article show that 
coefficient of $N$ in \autoref{eq:MY2} can be sharpened. In \autoref{sect:Section4-Inequalities} we make 
some predictions about possibly optimal versions of \autoref{thm:main2}.

\subsection{Acknowledgements} 
We like to thank Brian Lehmann and Henri Guenancia for answering our questions and helpful comments.

\section{Preliminaries}
\label{sect:Section2-Restriction}

In this section we review certain cones of divisors and curves that are needed for the rest of 
the article. 
Relevant notions of (slope) stability, orbifolds and Chern classes will also be introduced.  
 
 By a variety we mean a reduced, irreducible complex scheme of finite type.  
Given a variety $X$ of dimension $n$, by $\Div(X)$ and $\Pic(X)$ we denote the group of Cartier divisors 
and isomorphism classes of line bundles, respectively. 
$\N^1(X)$ denotes the N\'eron-Severi group consisting of numerical classes of elements 
of $\Div(X)$, i.e. $\N^1(X) = \Div(X) / \equiv$. 
We use $\Div(X)_{\bQ}$, $\N^1(X)_{\bQ}$ to 
denote $\Div(X)\otimes \bQ$ and $\N^1(X)\otimes \bQ$, respectively. 
The spaces $\Div(X)_{\bR}$ and $\N^1(X)_{\bR}$ are similarly defined. 
We recall that via intersection products $\N^1(X)$ is dual to the Abelian group 
of classes of curves $\N_1(X)$, which extends to cycle classes with rational or real coefficients 
(see \cite{Laz04-I}*{Sects.1.1,1.3} for more details). The notation $A^i(X)$ (and $A_i(X)$) 
refers to the Chow group of 
$i$-cocycles (resp. $i$-cycles) in $X$.

\begin{notation}
Given $D_i \in \Div(X)_{\bQ}$, for $1\leq i \leq n-1$, by $[D_1\cdot \ldots \cdot D_{n-1}]$ we denote the $1$-cycle 
in $N_1(X)_{\bQ}$ canonically defined by $D_i$'s, i.e. the image of 
$(c_1(\sO_X(D_1) \cdot \ldots \cdot c_1( \sO_X(D_n) )  \cap  [X] )\in A_1(X)\otimes \bQ$ under 
the cycle map, with $[X]$ denoting the fundamental cycle. 
\end{notation}
 
\subsection{Cones of curves and divisors, and stability notions}
Assuming that $X$ is projective, 
let $\NE^1(X)_{\bQ}\subset \N^1(X)_{\bQ}$ denote the convex cone of classes of effective $\bQ$-divisors 
and set $\overline{\NE^1}(X)_{\bQ}$ to be its closure, called the pseudo-effective cone. 

We define the notion of slope stability in the following general setting.

\begin{definition}[Slope stability]
\label{def:slope}
Given a torsion free sheaf $\sF$ on a smooth 
projective variety $X$ and $\gamma\in \N_1(X)_{\bQ}$, we define the slope $\mu_{\gamma}(\sF)$ of $\sF$ 
with respect to $\gamma$ by $\frac{1}{\rank(\sF)} c_1(\sF) \cdot \gamma$. 
A torsion free sheaf $\sE$ is said to be stable (or semistable) with 
respect to $0\neq \gamma$, if $\mu_{\gamma}(\sF) < \mu_{\gamma}(\sE)$ 
(resp. $\mu_{\gamma}(\sF) \leq \mu_{\gamma}(\sE)$), for every nontrivial
torsion free subsheaf $\sF \subset \sE$.
\end{definition}

Through the Harder-Narasimhan filtration, semistable sheaves form the 
building blocks of 
coherent, torsion free sheaves. But to ensure the existence of such (unique) filtrations, we generally need 
more assumptions on $\gamma$ in \autoref{def:slope}. In this article we require 
the existence of Harder-Narasimhan 
filtration under the assumption that $\gamma\in N_1(X)_{\bQ}$ is \emph{movable}.  
Thankfully, such filtrations are known to exist for such classes \cite{GKP15}*{Sect.2}.

\begin{definition}[Movable classes]
We say $\gamma\in \N_1(X)_{\bQ}$ is strongly movable, if there are 
a projective birational morphism $\pi: \wtilde X\to X$ and a set of 
ample divisors $H_1, \ldots , H_{n-1}$ on $\wtilde X$ 
such that $\gamma= \pi_* [ H_1 \cdot \ldots \cdot H_{n-1} ]$. The convex cone in $N_1(X)_{\bQ}$
generated by such classes is denoted by $\Mov_1(X)$. We call its closure $\overline\Mov_1(X)$ the movable cone. 
Nontrivial members of $\overline\Mov_1(X)$ are referred to as movable or mobile classes.
\end{definition}

As discussed in the introduction, for semistability to lead to a suitable Bogomolov inequality we need to work 
with a smaller set of $1$-cycles than those in $\overline\Mov_1(X)$. To do so we use 
the following definition. 

\begin{definition} 
Given an ample class $H \in \N^1(X)_{\bQ}$, we define 
$$
K^+_H(X) :=  \{  D \in \N^1(X)_{\bQ}  \;  \big|  \;  D^2 \cdot H^{n-2} >0 \; \text{and}  \;    D\cdot H^{n-1} >0   \} \;\;\;\; \subset \N^1(X)_{\bQ}. 
$$
Furthermore, we set 

$$
B^+_H(X) : = \{  [D\cdot H^{n-2}]  \in \N_1(X)_{\bQ}   \; \big|  \;   D\in K^+_H(X)    \}  \;\;\;\;  \subset  \N_1(X)_{\bQ}.
$$
\end{definition}

A key property of $K^+_H$ is its ``self-duality" in the sense that 
\begin{equation}
\label{eq:dual}
K^+_H(X) = \{  D\in \N^1(X)_{\bQ}  \; \big |  \;  D\cdot B \cdot H^{n-2} > 0, \text{for all} \; 0 \neq B \in \overline K^+_H(X) \} , 
\end{equation}
cf. \cite[p. 261]{Langer04a} and \cite{MR2665168}*{7.4}. In particular $B^+_H(X) \cup \{ 0\}$ forms a convex cone. 
The next theorem of Langer explains our interest in $B^+_H(X)$.

\begin{theorem}[\protect{\cite[Thm.~3.4]{Langer04a}}]
\label{thm:Langer}
For any torsion free sheaf $\sF$ of rank $r$ on $X$, satisfying the inequality 
$$
 \underbrace{\big( 2r c_2(\sF)  - (r-1) c_1^2(\sF)  \big)}_{\Delta_B(\sF)} \cdot H^{n-2} < 0
$$
there is a saturated subsheaf $ 0 \neq \sF' \subset \sF$ of rank $r'$ such that 
$$
\Big(  \frac{1}{r'} c_1(\sF')  -  \frac{1}{r}  c_1(\sF) \Big)  \in  K^+_H(X)\footnote{Here $c_1(\sF)$ 
is thought of as the dual of its image under the cycle map, represented 
by the reflexivization of $\det(\sF)$ (similarly for $c_1(\sF')$). } .
$$
\end{theorem}

\begin{remark}\label{rmk:Langer}
We note that by the self-duality property \autoref{eq:dual}, the subsheaf $\sF'$ in \autoref{thm:Langer} 
is a properly destabilizing subsheaf with respect to \emph{any} $\gamma\in B^+_H(X)$. That is, if 
$\sF$ is semistable with respect to some $\gamma\in B^+_H(X)$, then it verifies the Bogomolov 
inequality $\Delta_B(\sF) \cdot H^{n-2}\geq 0$. 
\end{remark}

\subsection{Orbifold sheaves and Chern classes} 
We follow the definitions and constructions of \cite[Sects.~2,3]{GT16}
in the generally simpler context of log-smooth pairs. 
We refer to \cite{Cam04}, \cite{MR2860268}, \cite{Taji16}, and \cite{CKT16} 
for more examples and details on pairs and associated 
notions of adapted morphisms. 

A \emph{pair} $(X,D)$ consists of a variety $X$ and $D = \sum d_i \cdot D_i\in \Div(X)_{\bQ}$, 
with $d_i = 1- \frac{b_i}{a_i} \in [0,1] \cap \bQ$, for some $a_i, b_i \in \bN$. 
A pair $(X,D)$ is said to be log-smooth, if $X$ is smooth and $D$ 
has simple normal crossing support. 
We say $(X,D)$ is (quasi-)projective, if $X$ is so.

We now recall a few 
basic notions regarding morphisms, sheaves and Chern classes encoding 
the fractional part of $D$ in $(X,D)$. 

\begin{definition}[Adapted morphisms]
Given a quasi-projective pair $(X,D)$, a finite, Galois and surjective morphism $f: Y \to X$ 
of schemes is called adapted (to $(X,D)$), if the following conditions are satisfied. 

\begin{enumerate}
\item $Y$ is normal and quasi-projective.
\item For every $D_i$, with $d_i\neq 1$, there are $m_i\in \bN$ and a reduced 
divisor $D'_i\subset Y$ such that $f^* D_i = (m\cdot a_i) \cdot D'_i$. 
\item The morphism $f$ is \'etale at the generic point of $\lfloor  D  \rfloor$.
\end{enumerate}
Furthermore, if $m_i=1$, for all $i$, we say $f$ is strictly adapted.
\end{definition}

\begin{example}\label{ex:Kaw}
Constructions of Bloch--Gieseker \cite{BG71} and Kawamata \cite[Prop.4.12]{Laz04-I} provide prime examples of 
strictly adapted morphisms with the following additional property: the ramification 
locus of $f$ is equal to $\supp( \lfloor D \rfloor + A )$, for some general member $A$ 
of a very ample linear system. Moreover, when $(X,D)$ is log-smooth, from their construction 
it follows that so is $(Y, (f^*D)_{\red})$.
\end{example}

\begin{notation}
\label{not:orb}
Given a log-smooth pair $(X,D)$, let $f: Y\to (X,D= \sum d_i \cdot D_i)$
be strictly adapted. Assume that $Y$ is smooth.
With $d_i= 1- (a_i/b_i)$, let $D^{ij}_Y$ be the collection of 
prime divisors in $\supp(f^* D_i)$ and define 
$$
\widehat D_Y^{ij} : = b_i \cdot D_Y^{ij} .
$$
Let $G:= \Gal(Y/X)$.
\end{notation}

\begin{definition}[Orbifold cotangent sheaf]
In the setting of \autoref{not:orb} we define the orbifold cotangent sheaf $\Omega^1_{(Y, f, D)}$ 
of $(X,D)$ 
with respect to $f$
by the kernel of the morphism 
$$
f^* \Omega^1_X \big(  \log \lceil D \rceil  \big) \longrightarrow  \bigoplus_{i, j(i)} \sO_{\widehat D_Y^{ij}} ,
$$
which is naturally defined using the residue map. 
\end{definition}
 
 We note that $\Omega^1_{(Y, f,D)}$ naturally has a structure of a $G$-sheaf \cite{CKT16} (see \cite[Def.~4.2.5]{MR2665168}
 for the definition). 
 Such objects are studied in a much more general setting (called orbifold sheaves) 
 in \cite[Subsect.~2.6]{GT16}. 
 
 \subsection{Orbifold Chern classes} 
 Let $f: Y\to (X,D)$ be a strictly adapted morphism for a  log-smooth pair $(X,D)$. Assume that 
 $Y$ is smooth and set $G:= \Gal(Y/X)$. Given a coherent $G$-sheaf $\sE$ on $Y$, we have 
 $c_i(\sE) \in A^i(Y)^G$. 
 Here $c_i(\cdot)$ 
 denotes the $i$th Chern class and $A^i(Y)^G$ the group of $G$-invariant, $i$-cocycles in $Y$. 
 We define the $i$th orbifold Chern class of $\sE$ by 
 \begin{equation}\label{eq:OC}
 \widehat c_i(\sE) : =  \frac{1}{|G|} \cdot \psi_i \big(  c_i(\sE)  \big) \;\;\;  \in A_{n-i}(X) \otimes \bQ, 
   \end{equation}
where $\psi_i$ is the natural map 
$\psi_i: A^i(Y)^G \otimes \bQ \to  A_{n-i}(X)\otimes\bQ$ defined by the composition of cap product 
with $[Y]$ and pushforward. 

With the above definition, when $X$ is projective, $\widehat c_i(\sE)$ defines 
a multilinear form on $N^1(X)_{\bQ}^{n-i}$. Furthermore, with $f$ being flat, from \cite[Thm.~3.1]{MR717614} it 
follows that $\psi$ is in fact a group isomorphism. Thus, similar to \cite{MR717614} 
(or \cite[Append.]{GT16}) we can use 
this isomorphism to equip $A_* (X)\otimes \bQ$ with a ring structure compatible 
with that of $A^*(Y)^G\otimes \bQ$. In this way, products of orbifold Chern classes 
can also be consistently defined in $A_*(X)\otimes \bQ$. 

One can check that for $G$-sheaves on $Y$, defined by pullback of sheaves on $X$, 
the above notion of orbifold Chern classes is consistent with the projection formula, 
when applicable.


\begin{remark}
\label{rk:Bog}
Let $f:Y \to (X,D)$ be strictly adapted to the log-smooth projective pair as in \autoref{ex:Kaw} and $H$ 
an ample $\bQ$-divisor. 
Let $\gamma\in B^+_H(X)$ and $\sF$ a torsion free, $G$-sheaf of rank $r$ on $Y$ that is $(f^*\gamma)$-semistable. 
Then, by \autoref{thm:Langer} and \autoref{rmk:Langer} we know that $\Delta_B(\sF) \cdot f^* H^{n-2}\geq 0$. 
Using \autoref{eq:OC} we can then deduce that 
\begin{equation*}
\label{eq:BOG}
\underbrace{\big( 2r \widehat c_2(\sF)  - (r-1) \widehat c_1^2  (\sF) \big)}_{\widehat \Delta_B(\sF)}  \cdot H^{n-2}  \geq  0 .
\end{equation*}
\end{remark}

\begin{notation}
Given a strictly adapted morphism $f: Y \to (X,D)$, with log-smooth $(Y, (f^*D)_{\red})$, we define 
\begin{equation}\label{not:CHERN}
\widehat c_i(X,D)  :  =   \widehat c_i( \Omega^1_{(Y,f,D)} ) \in  A_{n-i}(X)\otimes \bQ.
\end{equation}
We note that according to \cite[Prop.~3.5 and Ex. 3.3]{GT16} the cycle defined 
by $\widehat c_i(\Omega^1_{(Y,f,D)})$ 
is independent of the choice of $f$. In this light, the choice of notation in \autoref{not:CHERN} 
is unambiguous within the set of such morphisms. 
\end{notation}

\section{Constructing movable cycles in $B^+_H$ via Zariski decomposition}
\label{sect:Section3-Moving}
Our main goal in this section is to use Zariski decomposition on certain complete-intersection 
surfaces to construct global moving cycles in $B^+_H$, which is the content 
of \autoref{prop:star}.

\begin{proposition}
\label{prop:Moish}
Let $X$ be a smooth projective variety of dimension $n\geq 3$ and $H$ a very ample divisor. 
For a sufficiently large $m$, there are 
a (Zariski) dense subset $V_{\nl}\subseteq |mH|$ and a 
smooth complete 
intersection surface 
$$
S_{\nl} =   H_1\cap \ldots \cap H_{n-3}\cap A ,
$$
where each $H_i$ is a general members of $|H|$ and $A\in  V_{\nl}$, 
satisfying the following properties. 
\begin{enumerate}
\item \label{item:Pic} The restriction map $\Pic(X)\to \Pic(S_{\nl})$ is an isomorphism.
\item \label{item:Num} The isomorphism in \autoref{item:Pic} extends to an isomorphism 
$\N^1(X)\to \N^1(S_{\nl})$.
\end{enumerate}
\end{proposition}

\begin{proof}
By a repeated application of the Grothendieck-Lefschetz theorem \cite[Exp.~XII, Cor.~3.6]{SGA2} 
(see also \cite[Ex.~3.1.25]{Laz04-I} and further references therein) there are 
general members $H_i$ of $|H|$ such that $Y: = H_1\cap \ldots \cap H_{n-3}$ 
is a smooth projective threefold for which the natural map $\Pic(X)\to \Pic(Y)$ is
an isomorphism. Furthermore, for any $m\in \bN$, after shrinking $|mH|$, if necessary, 
the complete-intersection surface $S = Y\cap A$ is smooth, for every $A\in |mH|$.  
 
Now, consider the short exact sequence 
$$
\xymatrix{
0  \ar[r]  &  \sO_Y(K_Y)  \ar[r]  &  \sO_Y(K_Y +S) \ar[r]  &  \sO_{S}(K_{S})  \ar[r]  &  0  ,
}
$$
which is naturally defined by using adjunction. By the Kodaira vanishing $H^1(K_Y+S)=0$, 
the induced exact cohomology sequence partially reads 
$$
\xymatrix{
0  \ar[r]  &  H^0(K_Y)  \ar[r]^(0.4){i}  &  H^0(K_Y +S) \ar[r]  &  H^0(K_{S})  \ar[r]^{\alpha}  &  H^1(K_Y) \ar[r] &  0 .
}
$$
\begin{claim}\label{claim:Moish}
We have $h^{2,0}(Y) < h^{2,0}(S)$, if and only if  $i$ is a strict inclusion.
\end{claim}
\noindent 
\emph{Proof of \autoref{claim:Moish}.}
Noting that 
\[
h^{2,0}(Y) = h^{0,2}(Y) = h^2(\sO_Y) = h^1(K_Y)  \;  , \;
h^{2,0}(S) = h^{0,2}(S) = h^2(\sO_{S}) = h^0(K_{S})
\]
and the surjectivity of $\alpha$, we find that $h^{2,0}(Y)< h^{2,0}(S)$, if and only 
if $\ker(\alpha)\neq 0$. The rest now follows from a straightforward diagram chasing. \qed

Now, let $m$ be sufficiently large so that $i:H^0(K_Y)\to H^0(K_Y+S)$ is a strict injection. 
Thanks to a theorem of Moishezon~\cite[Thm.~7.5]{Moi67}, after removing a countable number of 
closed subschemes from $|mH|$, we find a subset $V_{\nl}\subseteq |mH|$ such that, 
for every $A\in V_{\nl}$ and $S_{\nl}: = Y \cap A$,  the natural map $\Pic(Y)\to \Pic (S_{\nl})$
is an isomorphism. 

For Item~\ref{item:Num}, we will keep the notations for the proof of Item~\ref{item:Pic}.
Again, since $\N^1(X) \cong N^1(Y)$ (see for example \cite[Ex.~3.1.29]{Laz04-I}), it suffices to prove 
$\N^1(Y)\cong \N^1(S_{\nl})$. As $S_{\nl}$ is reduced, we have a commutative diagram 
of long exact cohomology sequences arising from the two exponential sequences 
on $Y$ and $S_{\nl}$. In particular we have 
$$
\xymatrix{
H^1(Y, \sO^*_Y)  \ar[rr] \ar[d]  &&   H^2(Y; \bZ) \ar[d]  \\
H^1(S_{\nl} , \sO^*_{S_{\nl}})  \ar[rr]  &&    H^2(S_{\nl} ; \bZ)  . 
}
$$

Now, with the vertical arrow on the left being an isomorphism by Item~\ref{item:Pic} and 
the one on the right being an injection by the Lefschetz hyperplane theorem 
(\cite[Thm.~3.1.17]{Laz04-I} or \cite[2.3.2]{Voisin-Hodge2}), 
the isomorphism $\Pic(Y)\to \Pic(S_{\nl})$ descends 
to an isomorphism $\big(\Pic(Y)/\equiv \big) \to \big(\Pic(S_{\nl})/ \equiv\big)$, as required. 
\end{proof}

Before stating the application of \autoref{prop:Moish} that we need, we briefly review 
Nakayama's $\sigma$-decomposition.
\subsection{$\sigma$-decomposition}
According to \cite[Chapt.~III]{Nakayama04}, given a smooth projective variety and a 
big divisor $B\in \Div(X)_{\bR}$, for every prime divisor $\Gamma\subset X$, we define 
\begin{equation}\label{eq:Sigma}
\sigma_{\Gamma} (B)  :  =  \inf \{  \mathrm{mult}_{\Gamma} (B') \; \big|  \;  B' \equiv_{\bR}  B , \; B'\geq 0   \} , 
\end{equation}
and set $N_{\sigma}(B): = \sum\limits_{\text{prime} \; \Gamma} \sigma_{\Gamma}(B)\cdot \Gamma$. 
We further define $P_{\sigma}(B) : = B - N_{\sigma} (B)$. 
Now, let $D$ be a pseudo-effective divisor. For any ample divisor $A$, according 
to \cite[Lem.~III. 1.5]{Nakayama04} and \cite[Lem.~III. 1.7]{Nakayama04}, 
$\lim_{\epsilon \to 0^+} \sigma_{\Gamma}(D+ \epsilon A)$ exists and is independent of the choice of $A$.
We now set 
$$
\sigma_{\Gamma}(D) : = \lim_{\epsilon \to 0^+} \sigma_{\Gamma} (D+  \epsilon A) \footnote{When $D$ is big, this is consistent 
with \autoref{eq:Sigma}, cf.~\cite[Lem.~III. 1.7]{Nakayama04}.}.
$$
and define the \emph{negative part} of $D$ by 
\begin{equation}\label{eq:N}
N_{\sigma} (D): = \sum_{\text{prime}\; \Gamma} \sigma_{\Gamma} (D) \cdot \Gamma, 
\end{equation}
which  by \cite[Cor.~III.1.11]{Nakayama04} is a finite sum. 
We set $P_{\sigma}(D):= D -  N_{\sigma}(D)$ and call the decomposition $D= N_{\sigma}(D) + P_{\sigma}(D)$
the \emph{$\sigma$-decomposition} of $D$. We sometimes refer to $P_{\sigma}$ as the \emph{positive part} of this decomposition.


Next, we establish the fact that negative parts behave well under taking limits.
\begin{proposition}
\begin{equation}\label{eq:N2}
\lim_{\epsilon \to 0^+} N_{\sigma}(D+\epsilon A) = N_{\sigma} (D) . 
\end{equation}
\end{proposition}

\begin{proof}
First, let us observe that for any two pseudo-effective divisors $D_1$ and $D_2$ we have 
$$
\sigma_{\Gamma} (D_1 + D_2)  \leq  \sigma_{\Gamma}(D_1) + \sigma_{\Gamma} (D_2).
$$
Indeed, by definition, we have 
\begin{equation}\label{No1}
\sigma_{\Gamma} (D_1 + D_2) = \lim_{\epsilon \to 0^+} \sigma_{\Gamma} (D_1 + \frac{\epsilon}{2} A  + D_2 + \frac{\epsilon}{2} A ) .
\end{equation}
By \cite[Chapt.~III, p.~79]{Nakayama04} we have 
$$
\sigma_{\Gamma} (D_1 + \frac{\epsilon}{2} A  + D_2 + \frac{\epsilon}{2} A ) \leq \sigma_{\Gamma}(D_1 + \frac{\epsilon}{2}A) + 
\sigma_{\Gamma} (D_2 + \frac{\epsilon}{2}A).
$$

Therefore, after taking the limit and using (\ref{No1}), we find  
$$
\sigma_{\Gamma}(D_1 + D_2)  \leq \lim_{\epsilon \to 0^+} (\sigma_{\Gamma} (D_1 + \frac{\epsilon}{2}A) + \sigma_{\Gamma} (D_2 + \frac{\epsilon}{2}A)) = \sigma_{\Gamma} (D_1) + \sigma_{\Gamma} (D_2) . 
$$

It now follows that $N_{\sigma}(D_1 + D_2) \leq N_{\sigma}(D_1) + N_{\sigma}(D_2)$. 
In particular for every real $\epsilon>0$, we have:
$N_{\sigma} (D+ \epsilon A) \leq N_{\sigma}(D)$.

Next, define the set $S:= \{   \Gamma \; \text{prime} \; \big |  \; \sigma_{\Gamma}(D) \neq 0  \}$, which is finite by \cite[Cor.~III.1.11]{Nakayama04}.  By definition we now have 
\begin{align*}
N_{\sigma} (D)  & = \sum_{\Gamma\in S} \sigma_{\Gamma}(D)  \cdot \Gamma  \\
                    & =  \sum_{\Gamma \in S}  \lim_{\epsilon \to 0^+} \sigma_{\Gamma} (D+ \epsilon A)  \cdot \Gamma  \\
                    & = \lim_{\epsilon \to 0^+} ( \sum_{\Gamma\in S} \sigma_{\Gamma}(D+ \epsilon A)  \cdot \Gamma  \ ) . 
                    \end{align*}

On the other hand, for every $\epsilon>0$, we have 
$$
\sum_{\Gamma\in S} \sigma_{\Gamma} (D+\epsilon A) \cdot \Gamma \leq  \underbrace{\sum_{\Gamma} \sigma_{\Gamma} (D+\epsilon A) \cdot \Gamma}_{N_{\sigma}(D+\epsilon A)}  \leq  N_{\sigma} (D).
$$
Now, taking the limit $\epsilon \to 0^+$ finishes the proof.
\end{proof}

\begin{remark}
\label{rk:ZD}
For a pseudo-effective integral divisor $D$ on a smooth projective surface, the $\sigma$-decomposition 
of $D$ coincides with the usual Zariski decomposition, cf.~\cite[Rem.~1.17.(1)]{Nakayama04}. In particular, in this case 
$P_{\sigma}$ and $N_{\sigma}$ are $\bQ$-divisors. 
\end{remark}

\begin{corollary}
Let $S_{\nl}$ be the smooth projective surface in the setting of Proposition~\ref{prop:Moish}. 
\begin{enumerate}
\item \label{item:1} Let $\{ D_{\nl, m} \} \in \N^1(S_{\nl})_{\bQ}$ be a sequence 
converging to $D_{\nl}\in \N^1(S_{\nl})_{\bQ}$. Then, $\lim_{m\to \infty} D_m \equiv D$, 
where $D_m$ and $D$ are extensions of $D_{\nl,m}$ and $D_{\nl}$ 
under the isomorphism in Item~\ref{item:Num}.
\item \label{item:2} Assuming that $D_{\nl}\in \N^1(S_{\nl})_{\bQ}$ is pseudo-effective and $A\subset X$
is ample, let 
$N_{\frac{1}{m}}$ be the extension of 
$N_{\sigma}(D_{\nl}+ \frac{1}{m} A|_{S_{\nl}})$.
Then, we have $\lim_{m\to \infty}N_{\frac{1}{m}} = N$, where 
$N$ is the extension of $N_{\sigma}(D_{\nl})$.
\end{enumerate} 
\end{corollary}

\begin{proof}
Item~\autoref{item:1} immediately follows from Item~\autoref{item:Num}. Item~\autoref{item:2} follows from \autoref{item:1}
and \autoref{eq:N2}.
\end{proof}

\subsection{From positive parts to movable cycles in $B^+_H$}

\begin{notation}
Let $X$ be a quasi-projective variety of dimension $n$ with a very ample divisor $H$.  
For any subset $W\subseteq \prod^{n-2}_{j=1} |H|$, we use the notation $S\in W$ to say that 
$S = T_1\cap \ldots \cap T_{n-2}$ is a complete-intersection surface defined by some 
element $(T_1, \ldots, T_{n-2})\in W$.
\end{notation}

\begin{lemma}
\label{lem:move}
Let $X$ be a smooth projective variety of dimension $n\geq 3$ and $H$ a  very ample 
divisor. Let $D\in \N^1(X)_{\bQ}$ be a divisor class such that,  
for some $S\in \prod_{j=1}^{n-2} |H|$, the restriction 
$D|_S$ is nef. Then, 
\begin{enumerate}
\item \label{item:nef} 
after removing a countable number of closed subsets of $\prod_{j=1}^{n-2} |H|$, 
there is a (Zariski dense) subset $W^\circ_D\subseteq \prod_{j=1}^{n-2} |H|$ 
such that for every $S_{\eta}\in W^0_D$, the restriction $D|_{S_{\eta}}$ is nef, and 
\item \label{item:move} $[D\cdot H^{n-2}]\in \overline\Mov_1(X)_{\bQ}$.
\end{enumerate}
\end{lemma}

\begin{proof}
Let $i_{|H|}: X \hooklongrightarrow \mathbb P^l$ be the embedding defined by $|H|$, so that 
$i^*\sO_{\mathbb P^l}(1) \cong \sO_X(H)$. We set 
$$
\chi^{n-2}  \subset \bP^l  \times  \underbrace{ \prod_{j=1}^{n-2} \bP  \big(  H^0(\bP^l, \sO_{\bP^l}(1))  \big) }_{:= \Gamma_{n-2} }
$$
to be the (universal) family of complete-intersection surfaces cut out by hyperplanes 
in $| \sO_{\bP^l} (1)|$. More precisely, with $\{ a_{ij} \}_{0 \leq j \leq l}$ being the 
homogenous coordinates for the $i$th factor of $\Gamma_{n-2}$ and 
$\{ f_j \}_{0\leq j\leq l}$ a basis for $H^0(\sO_{\bP^l}(1))$, the variety $\chi^{n-2}$ 
is defined by the vanishing locus of 
$$
\big\{  g_i: =  \sum_{0\leq j\leq l}  a_{ij} f_j \big\}_{1\leq i \leq n-2}  . 
$$
Next, set $\chi^{n-2}_X$ to be the pullback of $\chi^{n-2}$ via the natural injection 
$$
\xymatrix{
X \times  \prod\limits_{j=1}^{n-2} |H| \;  \ar@{^{(}->}[rr]^{i_{|H|}\otimes \text{isom}}  &&  \bP^l \times  \Gamma_{n-2},
}
$$
with the isomorphism arising from the one for the vector spaces $H^0(X,\sO_X(H))\to H^0(\bP^l,\sO_{\bP^l}(1))$, as 
defined by $i_{|H|}$. Let 

$$
\xymatrix{
\chi^{n-2}_X  \ar[d]_(0.4){\sigma} \ar[rr]  &&    \chi^{n-2}  \ar[d] \\
X \times     \prod\limits_{j=1}^{n-2} |H|           \ar[rr]  &&     \bP^l \times  \Gamma_{n-2}  
}
$$
be the resulting commutative diagram. We define $\mu:= \pr_1\circ \sigma$ 
and use $f:= \pr_2\circ \sigma: \chi^{n-2}_X  \to \prod_{j=1}^{n-2}|H|$ 
to denote the induced proper morphism, with $\pr_k$ denoting the natural projection 
map to the $k$-th factor. 

Now, let $F_0$ be the fiber of $f$ corresponding to $S$. By assumption $\mu^* D |_{F_0}$ is nef. 
Therefore, thanks to openness of amplitude for projective morphisms  
(not necessarily flat) to Noetherian schemes  \cite[Prop.~1.41]{KM98},
we find that $\mu^*D|_{F_t}$ is also nef, for the very general fiber $F_t$, 
proving Item~\autoref{item:nef}.

For Item~\autoref{item:move}, let $B\in \overline{\NE^1}(X)_{\bQ}$ be any pseudo-effective 
class. With the above construction of $W_D^0$, we can find general, inductively constructed 
 $S'\in W^0_D$ such that 
$B|_{S'}$ is pseudo-effective. Therefore, 
$$
B\cdot D \cdot H^{n-2}  =  B|_{S'} \cdot D|_{S'} \geq 0.
$$

Now, since the inequality $B\cdot D\cdot H^{n-2}\geq 0$ holds for any $B\in \overline{\NE^1}(X)_{\bQ}$, 
this means that 
$$
[D\cdot H^{n-2}]  \in \big( \overline{\NE^1} (X)_{\bQ} \big)^*,
$$
that is the cycle $[D\cdot H^{n-2}]$ is dual\footnote{We say $\alpha\in \N_1(X)_{\bQ}$
is dual to $\overline{\NE^1}(X)_{\bQ}$, if, for every $D\in \overline{\NE^1}(X)_{\bQ}$ we have 
$D\cdot \alpha\geq 0$.} to the movable cone. On the other hand, by \cite[Thm.~0.2]{BDPP}
(and standard facts in convex geometry) we know that 
$\big(\overline{\NE^1}(X)_{\bQ}\big)^* = \overline\Mov_1(X)$, 
which finishes the proof.
\end{proof}

\begin{proposition}
\label{prop:star}
In the setting of \autoref{prop:Moish} let $P_{\nl}\in \N^1(S_{\nl})_{\bQ}$ be a nef and big class
with the extension $P \in \N^1(X)_{\bQ}$. Then, $[P\cdot H^{n-2}]\in \overline\Mov_1(X)\cap B^+_H(X)$.
\end{proposition}

\begin{proof}
Since $P_{\nl}$ is nef and big, we have $P^2\cdot H^{n-2}>0$ and $P\cdot H^{n-1}>0$, 
implying that $P\in K^+_H$, i.e. $[P\cdot H^{n-2}] \in B^+_H(X)$. 
The rest follows from \autoref{item:move}.
\end{proof}

\section{Miyaoka--Yau-type inequalities for non-uniruled varieties}
\label{sect:Section4-Inequalities}

We now proceed to the proof of \autoref{thm:main2}. 

\subsection{The general type case}
\label{gt}
Assume that $(X,D)$ is a log-smooth pair of log-general type. 
We may assume that $H$ is very ample 
and that the integer $m$ in \autoref{prop:Moish} is equal to $1$. 
Let $S_{\nl}$ be a smooth complete-intersection surface as constructed 
in \autoref{prop:Moish}. Note that for a suitable choice of $S_{\nl}$ we can ensure that 
$(K_X+D)|_{S_{\nl}}$ is big.  

Let $(K_X+D)|_{S_{\nl}} = P_{\nl} + N_{\nl}$ be the $\sigma$-decomposition 
 (which coincides with the Zariski decomposition by \autoref{rk:ZD}). 
 Let $P$ be the extension of $P_{\nl}$ under the isomorphism 
 $\N^1(X)_{\bQ}\to \N^1(S_{\nl})_{\bQ}$ and define $N:= (K_X+D)-P$.
 We note that with $P_{\nl}$ being nef and big, using \autoref{prop:star}, we have 
 $$
 \gamma: = [P\cdot H^{n-2}]  \in \overline\Mov_1(X) \cap B^+_H(X) .
 $$
 Therefore, thanks to \cite[Thm.~1.3]{CP16}, for every strictly adapted morphism $f: Y\to X$ as in 
 \autoref{ex:Kaw}, with the ramification locus given by 
 $\supp(\lfloor D\rfloor+ A)$, for some very ample divisor $A$, the orbifold cotangent sheaf
 $\Omega^1_{(Y, f , D)}$ is semipositive with respect to $f^*\gamma$. This means that for every 
 torsion free quotient $\sQ$ of $\Omega^1_{(Y,f,D)}$ 
 we have $c_1(\sQ)\cdot f^*\gamma \geq 0$. 
 
 Now, if $\Omega^1_{(Y,f,D)}$ is semistable with respect to $f^*\gamma$, then  
 by \autoref{thm:Langer} and \autoref{rmk:Langer} we have 
 $\widehat \Delta_B(\Omega^1_{(Y, f ,D)})\cdot H^{n-2}\geq 0$
 (see \autoref{rk:Bog}). Straightforward calculations, using the fact that $N^2\cdot H^{n-2}<0$, 
 then show that from this inequality 
 we can deduce $(3 \widehat c_2 - \widehat c_1^2)(X,D) \cdot H^{n-2}\geq N^2\cdot H^{n-2}$.
 We may thus assume that $\Omega^1_{(Y,f,D)}$ is \emph{not} semistable with respect to $f^*\gamma$. 
 

Define $G:= \Gal(Y/X)$. Let 
$$
(\sE_i)_{0\leq i\leq t} \subseteq  \Omega^1_{(Y,f,D)}  \;\;\; ,  \;\;\;\;\; \text{with $\sE_0 = 0$, $\sE_t = \Omega^1_{(Y,f,D)}$ and $t>1$,}
$$
be the increasing Harder-Narasimhan filtration with respect to $f^*\gamma= f^*(P\cdot H^{n-2})$.
For $1\leq i\leq t$, denote the torsion free, semistable quotients of this filtration by $\sQ_i:=\sE_i/\sE_{i-1}$, 
and set $r_i:=\rank(\sQ_i)$. As each $\sE_i$ is unique, it is equipped with a natural 
structure of a $G$-sheaf, and thus so is each $\sQ_i$.

According to \autoref{thm:Langer}, \autoref{rmk:Langer} and \autoref{rk:Bog}, for every $i$, we have 
\begin{equation}\label{eq:quotient}
\widehat \Delta_B(\sQ_i)  \cdot H^{n-2}  \geq 0.
\end{equation}
For the rest of this subsection we will closely follow the arguments of \cite[Prop.~7.1]{Miyaoka87}, 
adapting them to our setting by using the results in \autoref{sect:Section3-Moving}. 

\subsubsection{Step 1: A lower bound for $(3\widehat c^2 - \widehat c_1^2)(\sE_t)$ in terms of 
$(3\widehat c^2 - \widehat c_1^2)(\sE_1)$}
For $1\leq i \leq t$, define $\alpha_i\in \bQ$ by the equality 
$$
r_i \alpha_i =  \frac{\widehat c_1(\sQ_i)\cdot \gamma}{P^2 \cdot H^{n-2}}.
$$
Using $P\cdot N\cdot H^{n-2}=0$, this implies that 
\begin{equation}\label{eq:ONE}
\sum_{i=1}^{t} r_i\alpha_i  =  \frac{(K_X+D) \cdot \gamma}{P^2\cdot H^{n-2}}  =  \frac{(P+N)\cdot P\cdot H^{n-2}}{P^2\cdot H^{n-2}} = 1 .
\end{equation}

Moreover, as $\sE_t$ is semipositive with respect to $f^*\gamma$, we find $\alpha_t\geq 0$.
On the other hand, with $(\sE_i)_{0\leq i\leq t}\subseteq \Omega^1_{(Y,f,D)}$ being the Harder-Narasimhan filtration, 
by construction we have $\mu_{f^*\gamma}(\sQ_i) > \mu_{f^*\gamma}(\sQ_{i+1})$,  
which implies that 
\begin{equation}
\label{eq:positive}
\alpha_1 > \alpha_2 > \ldots > \alpha_t \geq 0 .
\end{equation}
Furthermore, with $W^0_P$ as in \autoref{lem:move},
we can find $S\in W_P^0$ such that 
the restriction of every $\sQ_i|_{\widehat S}$ is torsion free and that $\widehat S : = f^{-1}S$ is smooth.
Using  Item~\autoref{item:nef} we can then apply the Hodge index theorem for surfaces to conclude  
$$
c_1^2(\sQ_i|_{\widehat S}) \cdot (f^*P|_{\widehat S})^2   \leq    \big( c_1(\sQ_i|_{\widehat S})  \cdot  f^*P|_{\widehat S} \big)^2 .
$$
By writing this latter inequality in terms of orbifold Chern classes, we get 
$$
\big(  \widehat c_1^2 (\sQ_i) \cdot H^{n-2} \big)  ( P^2\cdot H^{n-2} )  \leq  \big(  \widehat c_1(\sQ_i) \cdot P\cdot H^{n-2}  \big)^2  ,
$$
which implies that 
\begin{equation}\label{eq:HI}
- \widehat c_1^2(\sQ_i) \cdot H^{n-2}  \geq - P^2 \cdot H^{n-2}  (r_i\alpha_i)^2 .
\end{equation}

We now consider 
\begin{equation}
\label{eq:first}
\big( 6 \widehat c_2  - 2 \widehat c_1^2\big) (\sE_t)  = \sum_{i=1}^t 6 \widehat c_2(\sQ_i) + 
                                6 \sum_{i<j} \widehat c_1(\sQ_i) \cdot \widehat c_1(\sQ_j)  - 2 \widehat c_1^2(\sE_t) . 
\end{equation}
Using
$$
\widehat c_1^2(\sE_t)  =  \sum_{i=1}^t \widehat c_1^2(\sQ_i)  +  2 \sum_{i<j} \widehat c_1(\sQ_i) \cdot \widehat c_1(\sQ_j) ,
 $$
we can then rewrite \autoref{eq:first} as 
\begin{align*}
\big( 6 \widehat c_2  - 2 \widehat c_1^2\big) (\sE_t)  &  =  \sum_{i=1}^t \big(  6 \widehat c_2 - 3 \widehat c_1^2  \big) (\sQ_i) + \widehat c_1^2(\sE_t) \\
                       & = 3 \sum_{i>1}^t \big( 2  \widehat c_2 - \widehat c_1^2 \big) (\sQ_i)  +  \big( 6 \widehat c_2 - 3 \widehat c_1^2 \big) (\sE_1)
                          +  \widehat c_1^2(\sE_t)  .
\end{align*}
Consequently, using the Bogomolov inequality \autoref{eq:quotient} we have 
$$
\big( 6 \widehat c_2  - 2 \widehat c_1^2\big) (\sE_t) \cdot H^{n-2}  \geq  \Big[  - 3 \sum_{i>1}^t \frac{1}{r_i}  \widehat c_1^2(\sQ_i) 
                      + ( 6 \widehat c_2 - 3 \widehat c_1^2) (\sE_1) + P^2  \Big]\cdot H^{n-2}  + N^2\cdot H^{n-2}. 
$$
By \autoref{eq:HI} it thus follows that 
$$
\big( 6 \widehat c_2  - 2 \widehat c_1^2\big) (\sE_t) \cdot H^{n-2}  \geq  \Big[  - 3 \sum_{i>1}^t \frac{1}{r_i}   P^2(r_i\alpha_i)^2 
                      + ( 6 \widehat c_2 - 3 \widehat c_1^2) (\sE_1) + P^2  \Big]\cdot H^{n-2}  + N^2\cdot H^{n-2}. 
$$
That is, we have 
\begin{equation}\label{STAR}
\big( 6 \widehat c_2  - 2 \widehat c_1^2\big) (\sE_t) \cdot H^{n-2}  \geq  \Big[  P^2 \big( 1 - 3 \sum_{i>1}^t r_i \alpha_i^2 \big)  + 
(6 \widehat c_2 - 3 \widehat c_1^2 ) (\sE_1)    \Big] \cdot H^{n-2} + N^2\cdot H^{n-2}. 
\end{equation}

\subsubsection{Step. 2: Analysis of $( 3 \widehat c_2 -  \widehat c_1^2  )(\sE_1)$ based on $\rank(\sE_1)$}
We now study the inequality \autoref{STAR} depending on $\rank(\sE_1)$.

\begin{claim}\label{claim:3}
If $\rank(\sE_1)\geq 3$, then $(3 \widehat c_2 - \widehat c_1^2)(\sE_t) \cdot H^{n-2}\geq \frac{1}{2}(N^2\cdot H^{n-2})$.
\end{claim}

\noindent
\emph{Proof of \autoref{claim:3}.}
Using \autoref{eq:quotient} for $\sE_1 = \sQ_1$ and \autoref{eq:HI} for $i=1$, from 
\autoref{STAR} it follows that 
$$
\big( 6 \widehat c_2  - 2 \widehat c_1^2\big) (\sE_t) \cdot H^{n-2}  \geq  \Big[  P^2 \big( 1 - 3 \sum_{i=1}^t r_i \alpha_i^2 \big)   
   \Big] \cdot H^{n-2} + N^2\cdot H^{n-2}. 
$$
On the other hand, by \autoref{eq:positive} we have $\alpha_1> \alpha_i$, for every $2\leq i\leq t$.
We thus find 
\begin{align*}
\big(  1 - 3 \sum_{i=1}^t r_i \alpha_i^2  \big)  P^2 \cdot H^{n-2}    &  \geq  \big( 1 - 3 \alpha_1 \sum_{i=1}^t  r_i \alpha_i  \big) P^2\cdot H^{n-2} \\
                                                    &  = \big( 1- 3\alpha_1\big) P^2\cdot H^{n-2} &  \text{by \autoref{eq:ONE}},
\end{align*}
so that 
$$
\big(6 \widehat c_2  - 2 \widehat c_1^2\big) (\sE_t)  \geq  \big( 1- 3\alpha_1 \big) P^2\cdot H^{n-2} + N^2\cdot H^{n-2}  .
$$
Using the assumption $r_1\geq 3$, the equality \autoref{eq:ONE}, and $\alpha_i\geq 0$, it follows that 
$\alpha_1\leq \frac{1}{3}$, i.e. $1- 3\alpha_1 \geq 0$, proving the claim. \qed

It remains to consider the case where $\rank(\sE_1)\leq 2$. To do so, we consider 
the short exact sequence 
\begin{equation}\label{eq:SE}
\xymatrix{
0  \ar[r] & \sE_1  \ar[r] &  \Omega^1_{(Y,f,D)} \ar[r] &  \sQ \ar[r]  & 0 ,
}
\end{equation}
with $\sQ$ being the torsion free quotient sheaf. 

\begin{claim}\label{claim:normal}
Let $W_P^0$ be as in \autoref{lem:move}. We can find $S\in W^0_P$ such that 
\begin{enumerate}
\item \label{i} the pair $\big(S, (D+A)|_S  \big)$ is log-smooth and thus so is $\big( \widehat S , D_{\widehat S}: = (f^*D_S)_{\red}   \big)$,
with $\widehat S:= f^{-1}S$, $D_S:= D|_S$, 
\item \label{ii} $\sE_1|_{\widehat S}$ is locally free, and 
\item \label{iii} the support of $\Omega^1_{(\widehat S, f, D_S)}\cap \sQ|_{\widehat S}$ 
is a proper subset of $\widehat S$, where $\Omega^1_{(\widehat S, f, D_S)}$ 
is the orbifold cotangent sheaf 
associated to $f|_{\widehat S}: \widehat S\to (S, D_S)$.
\end{enumerate}
\end{claim}

\noindent
\emph{Proof of Claim~\ref{claim:normal}.}
As $\big(\prod_{j=1}^{n-2} |H| \big) \backslash W_P^0$ consists of a union of only countable number of closed subsets, 
by a successive application of the Lefschetz hyperplane theorem, for a general member 
of $|H|$, items~\autoref{i} and \autoref{ii} are guaranteed to hold (note that $\sE_1$
is reflexive and thus locally free in codimension two). 
Same is true for Item~\autoref{iii} by the following observation:  
after removing a closed subscheme of $Y$, the surjection in \autoref{eq:SE} 
defines $\sQ^*$, locally analytically, as 
a sum of rank one foliations (trivially integrable). Therefore,  by choosing
$\widehat S$ transversal to the associated leaves, and using Nakayama's lemma, we can ensure that 
$\Omega^1_{(\widehat S, f, D_s)} \cap \sQ|_{\widehat S}$ has proper support.\qed

Now, by \autoref{claim:normal}, the composition 
\begin{equation}\label{compose}
\xymatrix{
\sE_{1, \widehat S}: = \sE_1|_{\widehat S} \ar[r] &  \Omega^1_{(Y, f, D)}|_{\widehat S} \ar[r]^{\alpha_N}  &  \Omega^1_{(\widehat S, f, D_S)}
}
\end{equation}
is generically injective, where $\alpha_N$ is naturally defined by the orbifold conormal 
bundle sequence. Since $\sE_1|_{\widehat S}$ is torsion free, it follows that the map \autoref{compose} 
is injective over $\widehat S$. We now consider two cases depending on $r_1$.

\subsubsection*{Case I: $\rank(\sE_1)=2$}
Using the injection \autoref{compose} 
$$
\sE_{1, \widehat S} \hooklongrightarrow \Omega^1_{(\widehat S, f, D_S)} \subseteq  \Omega^1_{\widehat S} (\log D_{\widehat S}), 
$$
according to \cite[Rk.~4.18]{Mi84}, we either have 
$(3 c_2 -  c_1^2) (\sE_{1, \widehat S}) \geq 0$, or $\kappa(\sE_{1, \widehat S}):= \kappa(\det \sE_{1, \widehat S})<0$.

If $(3 \widehat c_2 - \widehat c_1^2)(\sE_1) \cdot H^{n-2}\geq 0$, then by \autoref{STAR} we have 
\begin{align}\label{casei}
\big( 6 \widehat c_2 - 2 \widehat c_1^2\big)  \cdot H^{n-2}  &  \geq  \big[  P^2( 1- 3 \sum_{i>1}^t r_i\alpha_i^2 )  - \widehat c_1^2(\sE_1) \big]\cdot H^{n-2}  +  N^2\cdot H^{n-2} \nonumber \\
                  &  \geq  \big[   P^2(1 - 3\sum_{i>1}^t r_i\alpha_i^2 ) - P^2(r_1\alpha_1)^2 \big]\cdot H^{n-2} + N^2\cdot H^{n-2} & \text{by \autoref{eq:HI}} \nonumber \\
                  & = \Big( P^2\big(   1- 4\alpha^2  - 3\sum_{i>1}^t r_i\alpha_i^2  \big)  \Big) \cdot H^{n-2}  + N^2\cdot H^{n-2} & \text{as $r_1$=2} 
                  \nonumber \\
                  & \geq P^2 \big(   1 - 4\alpha_1^2 - 3\alpha_2 \sum_{i>1}^t r_i\alpha_i \big) \cdot H^{n-2}  + N^2\cdot H^{n-2} & \text{by \autoref{eq:positive}}  \nonumber \\
                  & =  P^2\big(  1 - 4\alpha_1^2 - 3\alpha_2(1- 2\alpha_1)  \big)\cdot H^{n-2} + N^2\cdot H^{n-2} & \text{by \autoref{eq:ONE}} \nonumber \\
                  & = P^2(1- 2\alpha_1) ( 1+ 2\alpha_1 - 3\alpha_2 ) \cdot H^{n-2}  + N^2\cdot H^{n-2}  .
\end{align}
On the other hand, using \autoref{eq:positive} we have 
\begin{align*}
3\alpha_2 \leq 2 \alpha_2 + r_2\alpha_2  & < 2 \alpha_1 + r_2 \alpha_2  \\
                                        &  \leq 1  & \text{by \autoref{eq:ONE}} ,
\end{align*}
implying that $1- 3\alpha_2 \geq 0$. Furthermore, again by \autoref{eq:ONE}, we have $2\alpha_1\leq 1$, i.e. $1- 2\alpha_1 \geq 0$. 
Going back to \autoref{casei} we now find 
\begin{align*}
\big( 6 \widehat c_2 - 2 \widehat c_1^2\big)  \cdot H^{n-2}   &  \geq 2 \alpha_1 \cdot P^2 \cdot H^{n-2} + N^2\cdot H^{n-2} \\
                                   & \geq N^2 \cdot H^{n-2},
\end{align*}
establishing our desired inequality. 

We now assume that $\kappa(\sE_{1, \widehat S})<0$. 
As $S\in W^0_P$, the restriction $P|_S$ is nef and thus so is 
$f^*P|_{\widehat S}$. Moreover, as $\alpha_1>0$, we have 
\begin{equation}\label{nef}
c_1(\sE_{1, \widehat S})\cdot f^*P|_{\widehat S}>0
\end{equation}
Using Riemann-Roch we thus get $\widehat c_1^2(\sE_{1, \widehat S})\leq 0$.
Going back to \autoref{STAR} we get 
\begin{align*}
\big(  6 \widehat c_2 - 2 \widehat c_1^2 \big) (\sE_t) \cdot H^{n-2} &  \geq  \big[ P^2(1- 3\sum_{i>1}^t r_i\alpha_i^2 ) + 
  3  ( 2\widehat c_2 - \widehat c_1^2 ) (\sE_1) \big] \cdot H^{n-2} + N^2\cdot H^{n-2}  \\
 \text{by \autoref{eq:quotient}}  \;\;\;\;\;\;\;\;\;\;\;\; \;\;\;\;\;\; & \geq  \big[ P^2(1- 3\sum_{i>1}^t r_i\alpha_i^2 ) + 
  3  ( \frac{-1}{r_1}) \widehat c_1^2(\sE_1) \big] \cdot H^{n-2} + N^2\cdot H^{n-2}  \\
  &  \geq P^2\big( 1 - 3 \sum_{i>1}^t r_i\alpha_i^2  \big)  \cdot H^{n-2}   +  N^2\cdot H^{n-2}   \\
  &  \geq P^2\big(  1- 3\alpha_2 \sum_{i>1}^t r_i\alpha_i \big)  \cdot H^{n-2}   + N^2\cdot H^{n-2} \\
  &  = P^2 \big(  1 -3\alpha_2 (1-2\alpha_1)   \big) \cdot H^{n-2} + \N^2\cdot H^{n-2} \;\;\;\;\;\;\;\;\;\;\;\;  \text{by \autoref{eq:ONE}} \\
  \text{by \autoref{eq:positive}} \;\;\;\;\;\;\;\;\;\; \;\;\;\;\;\;\;\; & \geq  P^2 \underbrace{\big(   1 - 3\alpha_1(1 - 2\alpha_1)  \big)}_{1- 3\alpha_1 + 6\alpha_1^2 = 6 (\alpha_1-\frac{1}{4})^2  + \frac{5}{8} \;\;\;\;\;\; > \; 0 } \cdot \; H^{n-2}  + N^2\cdot H^{n-2}   
\end{align*}

\subsubsection*{Case II: $\rank(\sE_1)=1$} 
Again, by using the injection \autoref{compose}, we have 
$\sE_{1, \widehat S} \hookrightarrow \Omega^1_{\widehat S}(\log D_{\widehat S})$. 
Therefore, thanks to Bogomolov-Sommese vanishing \cite{Bog79}, \cite{Miy77}, \cite{SS1985} 
(see also \cite{ProcHodge} and \cite{EV92} for generalizations), 
we have $\kappa(\widehat S, \sE_{1, \widehat S})\leq 1$. 
On the other hand, we have the inequality \autoref{nef}. With $f^*P|_{\widehat S}$ being nef, 
using Riemann-Roch, we thus find $\widehat c_1^2(\sE_{1, \widehat S})\leq 0$. 
Moreover, as $\rank(\sE_{1, \widehat S}) =1$, we have $c_2(\sE_{1, \widehat S})=0$. Now, going back to 
\autoref{STAR} we get 
\begin{align*}
\big(  6 \widehat c_2 - 2 \widehat c_1^2 \big)  (\sE_t) \cdot H^{n-2}  &  \geq  P^2 \big( 1 - 3\sum_{i>1}^t r_i\alpha_i^2  \big) \cdot H^{n-2} 
    + N^2\cdot H^{n-2}  \\
      &  \geq  P^2 \big(  1 -  3\alpha_1 \sum_{i>1}^t r_i\alpha_i  \big)  \cdot H^{n-2}  + N^2\cdot H^{n-2} & \text{by \autoref{eq:positive}} \\
      & = P^2 \big(  1 - 3\alpha_1 (1- \alpha_1)  \big) \cdot H^{n-2}  + N^2 \cdot H^{n-2} & \text{by \autoref{eq:ONE}} \\
      &  = P^2 \Big[  3 \big( (\alpha_1 - \frac{1}{2})^2 - \frac{1}{4}   \big)   +1    \Big] \cdot H^{n-2}  + N^2\cdot H^{n-2}  \\
      & = P^2 \big(  3 (\alpha_1 - \frac{1}{2})^2 + \frac{1}{4} \big) \cdot H^{n-2}  + N^2\cdot H^{n-2}  \\
      & \geq N^2\cdot H^{n-2}  ,
\end{align*}
which finishes the proof of the log-general type case.

\subsection{The pseudo-effective case}
Assuming that $K_X+D$ is pseudo-effective, for any very ample divisor $A$ and $m\in \bN$, we consider 
the pair $(X, D+ \frac{1}{m}A)$. 
We may assume that $H$ in the setting of \autoref{thm:main2} is very ample. 
For $1\leq i \leq n-3$, let $H_i\in |H|$ 
be general members such that $S_{\nl}$ is a complete-intersection surface 
as in \autoref{prop:Moish}
and that $(K_X+D)|_{S_{\nl}}$ is pseudo-effective.

By the general type case we know that 
\begin{equation}\label{GT}
\big(  3 \widehat c_2 - \widehat c_1^2  \big) (X,D+ \frac{1}{m} A) \cdot H^{n-2}  \geq  N^2_{\frac{1}{m}} \cdot H^{n-2}  ,
\end{equation}
where $N_{\frac{1}{m}}$ denotes the extension of $N_{\sigma}\big((K_X+D+\frac{1}{m} A )|_{S_{\nl}}\big)$. 
Using the continuity of orbifold Chern numbers \cite[Prop.~3.11]{GT16} and Item~\autoref{item:2}, 
from \autoref{GT} it follows that 
$$
\big(  3\widehat c_2 - \widehat c_1^2   \big) (X,D)  \cdot H^{n-2}  \geq  N^2\cdot H^{n-2},
$$
with $N$ being the extension of $N_{\sigma}\big((K_X+D)|_{S_{\nl}}\big)$.

\subsection{Concluding remarks}
As is evident from the proof of \autoref{thm:main2}, the inequality \autoref{eq:MY2} can be sharpened 
to 
$$
\big( 3\widehat c_2(X,D)  - \widehat c_1^2(X,D) \big) \cdot H^{n-2}  \geq \frac{1}{2} N^2\cdot H^{n-2} .
$$
It would be interesting to know, if this can be improved further by the inequality 
\begin{equation}\label{FINAL}
\big( 3 \widehat c_2(X,D)  -  \widehat c_1^2(X,D)  \big) \cdot H^{n-2}  \geq  \frac{1}{2} \big(1  - \frac{3}{n}  \big) N^2 \cdot H^{n-2}  .
\end{equation}
We note that \autoref{FINAL} coincides with \cite[Rk.~4.18]{Mi84}, when $\dim =2$, and the claimed inequality 
in \cite[p.~498]{LuMi} in higher dimensions. 
We refer to \cite[Rem.~8.2]{RT16} for a brief discussion of gaps in the proof of the latter inequality.

\bibliographystyle{bibliography/skalpha} 
\bibliography{bibliography/general}

\end{document}
